\documentclass[12pt]{amsart}
\usepackage{url}
\usepackage{mathpazo}
\usepackage[top=1.25in, bottom=1.25in, left=1.05in, right=1.05in]{geometry}

\def\P{\mathbb{P}}
\def\E{\mathbb{E}}
\newcommand{\N}{\mathbb{N}}

\newcommand{\X}{\mathcal{X}}

\newcommand{\VCdim}{\mathrm{VCdim}}

\newcommand{\ind}[1]{\mathbf{1}\big\{#1\big\}} 
\usepackage{enumerate}

\usepackage[style=numeric, isbn=false, doi=false, url=false, backend=biber]{biblatex}
\DeclareFieldFormat[article,periodical]{number}{}
\renewbibmacro{in:}{%
 \ifentrytype{article}{}{%
 \printtext{\bibstring{in}\intitlepunct}}}
\DeclareFieldFormat[article,incollection]{pages}{#1}%
\usepackage{amsthm,amsmath,amssymb,graphicx}
\usepackage[colorlinks=true,citecolor=black,linkcolor=black,urlcolor=blue]{hyperref}

\theoremstyle{plain}
\newtheorem{defn}{Definition}[section]

\newtheorem{prop}[defn]{Proposition}
\newtheorem{lem}[defn]{Lemma}
\newtheorem{thm}[defn]{Theorem}
\newtheorem{cor}[defn]{Corollary}

\begin{document} 
\title[VC-dimension, primes, and boosting]{The VC-dimension of a class of multiples of the primes, \\ and a connection to AdaBoost}
\author{Andrew M. Thomas} 
\address{Center for Applied Mathematics \\ Cornell University}
\email{amt269@cornell.edu} 

\begin{abstract}
We discuss the VC-dimension of a class of multiples of integers and primes (equivalently indicator functions) and demonstrate connections to prime counting functions. Additionally, we prove limit theorems for the behavior of an empirical risk minimization rule as well as the weights assigned to the output hypothesis in AdaBoost for these ``prime-identifying'' indicator functions, when we sample $m_n$ i.i.d. points uniformly from the integers $\{2, \dots, n\}$. 
\end{abstract}

\date{}
\maketitle

\allowdisplaybreaks

\section{The VC-dimension of a class of multiples of primes}

Let $\N$ denote the positive integers. Recall that for functions $f, g: \N \to (0, \infty)$ we say $f(n) = \omega(g(n))$ if $\lim_{n \to \infty} f(n)/g(n) = \infty$. Consider a learning problem with domain $\X = \{n \in \N: n \geq 2\}$, label set $\{0, 1\}$ and hypothesis class $\mathcal{H}' := \{h_p: p \in \X, p \text{ is prime}\},$ where 
\[
h_p(x) = 
\begin{cases}
1 &\mbox{ if } x \leq p \text{ or } p \nmid x, \\
0 &\mbox{ otherwise}. 
\end{cases}
\]

\noindent We have $h_p(x) = 0$ if and only if $x$ is divisible by $p$---i.e., $p \mid x$ and $x > p$. Note that if $x$ is prime then $h_p(x) = 1$. Another characterization is that $h_p(x) = 0$ if and only if $x \in p\N\setminus \{p\} = \{2p, 3p, 4p, \dots\}$. In this way, our efforts can be seen as assessing the sample complexity of hypothesis class of indicator functions associated to the Sieve of Erastosthenes. A notable study containing results related to our investigations of VC-dimension is \cite{learning_integer}.

A class of indicator functions $\mathcal{G} = \big\{g: \X \to \{0,1\}\big\}$ is said to \emph{shatter} a set $C = \{c_1, \dots, c_\ell\} \subset \X$ if the cardinality 
\[
\Big | \big\{(g(c_1), \dots, g(c_\ell)): g \in \mathcal{G}\big\}\Big | = 2^{\ell}. 
\]
The \emph{VC-dimension} of $\mathcal{G}$, denoted $\VCdim(\mathcal{G})$ is the size of the largest set that $\mathcal{G}$ shatters. See \cite{shaishai} for more details. To begin, we will show that $\mathcal{H}'$ can shatter a set of size 2. Let $C = \{6, 10\}$. Then $(h_2(6), h_2(10)) = (0,0)$, $(h_3(6), h_3(10)) = (0, 1)$, $(h_5(6), h_5(10) = (1, 0)$ and $(h_{7}(6), h_7(10)) = (1,1)$. Notice there are $2^2-1 =3$ unique prime factors in $6=2\cdot3$ and $10=2\cdot 5$. Thus, $\VCdim(\mathcal{H}') \geq 2$. We will now show that $\VCdim(\mathcal H') \geq n$ for any $n \geq 3$. Let us enumerate the first $2^n$ primes $p_1 < p_2 < \cdots < p_{2^n}$ and recognize that 
\[
2^n  = \sum_{k=1}^n \binom{n}{k}.
\]
Each $c_i$ will be the product of $2^{n-1}$ primes. Label all $2^n$ subsets of $\{1, \dots, n\}$ in some fashion, e.g. $A_1, \dots, A_{2^n}$. 
We may now form $c_i$, $i=1,\dots, n$ as follows:
\begin{equation} \label{e:ci}
c_i := \prod_{k=1}^{2^n} p_k^{\mathbf{1}_{A_k}(i)}.
\end{equation}
We may take a second to see that 
\[
\sum_{k=1}^{2^n} \mathbf{1}_{A_k}(i) = \sum_{j=0}^{n-1} \binom{n-1}{j} = 2^{n-1}, 
\]
as $i$ must be in each set and there are $\binom{n-1}{j}$ other options for a set containing index $i$ of cardinality $j+1$. For simplicity, denote $h_{p_k} = h_k$, $k = 1, \dots, 2^n$. We aim to prove that $h_k(c_i) = 1 - \mathbf{1}_{A_k}(i)$. By definition, 
$h_k(c_i) = 0$ if and only if $p_k \mid c_i$ and $c_i > p_k$. However, when $n \geq 2$ then $c_i$ is the product of $2^{n-1} \geq 2$ primes, so that if $p_k \mid c_i$, then $c_i > p_k$. Hence, $h_k(c_i) = 0$ if and only if $p_k \mid c_i$. 
%
%
However, $p_k \mid c_i$ if and only if $i \in A_k$, as then and only then will $p_k$ appear in $c_i$'s factorization. Thus, $h_k(c_i) = 0$ if and only if $\mathbf{1}_{A_k}(i) = 1$, or $h_k(c_i) = 1-\mathbf{1}_{A_k}(i)$. We have thus proved that $\mathcal{H}'$ shatters a set of size $n$ for all $n \geq 2$, hence 
\[
\VCdim(\mathcal{H}') = \infty.
\]
All finite classes have finite VC-dimension, hence the infinitude of the set of primes is equivalent to $\VCdim(\mathcal H') = \infty$. 
Thus, we have established a theorem which relates Euclid's second theorem to the VC-dimension of the class $\mathcal{H}'$.  
\begin{thm}\label{t:main}
The following statements are equivalent: \vspace{11pt}
\begin{enumerate}[\it (i)]
	\item The set of primes is infinite
	\vspace{11pt}
	\item $\VCdim(\mathcal H') = \infty$
\end{enumerate}
\end{thm} 

We have two interesting corollaries as a result of the above. We state the first here. 
\begin{cor} \label{c:pshat}
Suppose that $\mathcal H'$ shatters a set $\{c_1, \dots, c_n\}$. Then $c_i$ is divisible by the product of $2^{n-1}$ distinct primes and $\prod_{i=1}^n c_i$ has at least $2^n-1$ distinct prime factors
\end{cor}
\begin{proof}
If $\mathcal H'$ shatters $\{c_1, \dots, c_n\}$, then we have for each subset $A_k$ of $\{1, \dots, n\}$ that 
\[
\big(h(c_1), \dots, h(c_n)\big) = \Big( 1-\mathbf{1}_{A_k}(1), \dots, 1-\mathbf{1}_{A_k}(n) \Big),
\]
for some $h \in \mathcal H'$, which we assign a unique prime $p_{j_k}$. Therefore, $h(c_i) = 0$, i.e. $p_{j_k} \mid c_i$ and $c_i > p_{j_k}$, if and only if $i \in A_k$. As a result of the shattering,
\begin{equation} \label{e:max_prime}
c_i > \max_k p_{j_k}^{\mathbf{1}_{A_k}(i)},
\end{equation}
and 
\begin{equation}\label{e:prod_prime}
\prod_{k=1}^{2^n} p_{j_k}^{\mathbf{1}_{A_k}(i)} \mid c_i.
\end{equation}
The last statement in the corollary follows from the fact that there exists an $i$ and a $k$ such that $i \in A_k$ for all but one subset of $\{1, \dots, n\}$---the empty set.
\end{proof}

Note that \eqref{e:max_prime} is automatically satisfied by \eqref{e:prod_prime}, as all primes are at least 2 and $\prod_{k=1}^{2^n} p_{j_k}^{\mathbf{1}_{A_k}(i)}$ consists of precisely $2^{n-1}$ primes. Hence, if there exists primes $p_{j_1}, \dots, p_{j_{2^n}}$ such that \eqref{e:prod_prime} holds, then \eqref{e:max_prime} is satisfied, so that $h_{p_{j_k}}(c_i) = 0$ for all $k$ such that $i \in A_k$.  

Now suppose that we have the class $\mathcal{H}'_{\leq n} = \{ h_p: \, p \leq n\}$, and a set $C$ of size $j(n) = \lfloor \log_2 \pi(n) \rfloor$, where $C = \{c_1, \dots, c_{j(n)}\}$ and each $c_i$ defined as at \eqref{e:ci}. We may shatter $C$ as there are $2^{j(n)} \leq \pi(n)$ primes available to us. Additionally $|\mathcal{H}'_{\leq n}| \leq \pi(n)$. Thus, we have an interesting corollary to Theorem~\ref{t:main}. 
\begin{cor}
For the hypothesis class $\mathcal{H}'_{\leq n} = \{ h_p: \, p \leq n\}$, we have
\[
\VCdim(\mathcal{H}'_{\leq n}) = \lfloor \log_2 \pi(n) \rfloor.
\]
\end{cor}

One thing that we may conclude from Theorem~\ref{t:main} is that the class $\mathcal{H}'$ of functions which cross out the multiples of primes (resp. positive integers greater than 1), is not \emph{uniformly learnable} (in a probably approximately correct sense)---cf. \cite{shaishai}. 
\section{The hopelessness of boosting the primes}

The original consideration of the problem described above was under the pretense that one might be able to ``learn'' the primes if one could leverage a knowledge of divisibility by a certain integer. That is, could you use some combination of elementary heuristics in $\mathcal{H}'$ to determine whether or not a number is prime? The premise of using simple ``rules of thumb'' underlies the idea behind AdaBoost (cf. \cite{boosting} for a detailed account). How good is the output of AdaBoost using a base hypothesis class $\mathcal{H}'$? As it turns out, when you have a sufficiently large uniform random sample from $\X_n$, AdaBoost is powerless to try to predict which numbers are primes and which aren't. The primes are simply too rich for $\mathcal{H}'$ to ``understand''.

Before continuing to the results, let us show that there is 1) no loss in generality in considering only the divisibility by primes and that 2) weak learning isn't possible. We address the first point by considering the larger hypothesis class $\mathcal{H} = \{h_d: d \in \X\}$ for $h_d$ defined analogously to $h_p$ above, though we allow our $d$ to be \emph{any} integer greater than or equal to 2, not just a prime. If our instances $X_1, \dots, X_m$ are labelled by whether or not they are prime---i.e. $r(x) = 1$ if $x$ is prime and $0$ otherwise---then any algorithm that returns an \emph{empirical risk minimization} (ERM) hypothesis will return only a hypothesis in $\mathcal{H}'$. We will denote our sample as $S = (X_1, r(X_1)), \dots, (X_m, r(X_m))$. Choose nonnegative constants $\mathbf{a} = (a_1, \dots, a_m)$ such that $\sum_{i=1}^m a_i  =1$, and define the weighted empirical risk of $h_d$ as 
\begin{align*}
L_{\mathbf{a}}(S, d) &:= \sum_{i=1}^m a_i \ind{ h_d(X_i) \neq r(X_i) } \\
&= \sum_{i: \, r(X_i) = 0} a_i \ind{ h_d(X_i) = 1 }
\end{align*}
as only composite numbers factor into the error. For ease of exposition, define
\[
D(S, d) := \sum_{i: \, r(X_i) = 0} a_i \ind{X_i \text{ is divisible by } d}\ind{X_i > d},
\]
so that $L_{\mathbf{a}}(S, d) = 1-D(S,d)$. Now define
\[
d_S = \min\Big \{ d\in \X: D(S,d) = \max_{k \in \X} D(S, k) \Big\}.
\]
It is straightforward to show that $h_S := h_{d_S}$ is an ERM rule, as the training error for $h_d$ with respect to the prime labeling function is $1-D(S, d)$. We aim to show that $d_S$ is prime. To show this, suppose that $p \mid d_S$. Then $p \leq d_S$ and if $d_S \mid x_i$ and $x_i > d_S$ then $p \mid x_i$ and $x_i > p$. Thus
\[
D(S, d_S) \leq D(S, p), 
\]
but since $D(S, d_S)$ is maximal, $D(S, d_S) = D(S, p)$, and hence $p \geq d_S$, thus $p = d_S$. However, and perhaps no surprise due to the VC-dimension of $\mathcal{H}'$, the error $L_{\mathbf{a}}(S, d)$ can be arbitrarily large (depending on $\mathbf{a}$). Consider a sample of $m$ values consisting of products of consecutive primes, e.g. $x_1 = p_1p_2, x_2 = p_3p_4, \dots, x_m = p_{2m-1} p_{2m}$. All of these numbers are composite so that $L_{\mathbf{a}}(S, d) \geq 1-\max_i a_i$. With this established, we cannot necessarily find what is called a ``weak learner'' $h_d$ such that $L_{\mathbf{a}}(S, d) \leq 1/2 - \gamma$, where $\gamma > 0$. There is clearly no hope of boosting this hypothesis class to discover the primes---see \cite{boosting} for more information on weak learning. 

There is something more quantitatively forceful we can say about the hopelessness of our situation. We consider the stochastic distribution of the weights derived by running AdaBoost with $\mathcal{H}'$ (or $\mathcal{H}$) as the base hypothesis class and where the underlying distribution is uniform on $\{2, \dots, n\}$. Let $\X_n := \{2, \dots, n\} \subset \X$ and for a probability distribution $\mathcal{D}$ define the generalization error to be $L_{\mathcal{D}}(h_d) := \P(h_d(X) \neq r(X))$ and the empirical risk $L_S(h_d) := L_\mathbf{a}(S, d)$ where $\mathbf{a} = (1/m, \dots, 1/m)$. 

\begin{prop} \label{e:wt_conv}
Consider the hypothesis class $\mathcal{H}$ or $\mathcal{H}'$. Suppose the $\mathcal{D}_n$ is the uniform distribution on $\{2, \dots, n\}$ and we have the prime labeling function $x \mapsto r(x)$ for the instances $X_1, \dots, X_{m_n}$ in our sample $S_n$. Then if $m_n = \omega\big( (\log n)^2 \big)$, we have the following convergence in probability,
\[
L_{\mathcal{D}_n}(h_{S_n}) \overset{P}{\to} 1/2,
\]
as $n \to \infty$. Furthermore, if we have $m_n = \omega\big( (\log n)^3 \big)$, then we have convergence of the generalization error
\[
L_{\mathcal{D}_n}(h_{S_n}) \overset{\mathrm{a.s.}}{\to} 1/2, \quad n \to \infty,
\]
where $\overset{\mathrm{a.s.}}{\to}$ denotes almost sure convergence, or convergence with probability 1.
\end{prop}
\begin{proof}
We begin by denoting 
\[
t(n) := \big | \{x \in \mathcal{X}: x \leq n, \, x \text{ is even}\} \big | = 
\begin{cases}
(n-1)/2 &\mbox{if } n \text{ is odd }, \\
n/2 &\mbox{if } n \text{ is even}.  
\end{cases}
\]
It follows that $|t(n)/(n-1) -1/2| \leq \frac{1}{2(n-1)}$. Notice that 
\begin{align*}
L_{\mathcal{D}_n}(h_2) &= \P(h_2(X) = 1, r(X) = 0) \\
&= \P(f(X) = 0) - P(h_2(X) = 0, r(X) = 0) \\
&= \P(f(X) = 0) - P(h_2(X) = 0) \\
&= 1 - \frac{\pi(n)}{n-1} - \frac{t(n) - 1}{n-1},
\end{align*}
hence $|L_{\mathcal{D}_n}(h_2) - 1/2| \leq (3/2 + \pi(n))/(n-1) \to 0$, $n \to \infty$. So there exists some $N$ such that if $n \geq N$ then 
\[
|L_{\mathcal{D}_n}(h_2) - 1/2| \leq (3/2 + \pi(n))/(n-1) \leq \epsilon/2.
\]
Elementary union bounds yield that 
\begin{align*}
&\P( |L_{\mathcal{D}_n}(h_{S_n}) - 1/2| > \epsilon) \\
&\qquad \leq \P( |L_{\mathcal{D}_n}(h_{S_n}) - L_{\mathcal{D}_n}(h_2)| > \epsilon/2) + \ind{  |L_{\mathcal{D}_n}(h_2) - 1/2| > \epsilon/2 } \\
&\qquad \leq \P( d_{S_n} \neq 2 ),
\end{align*}
for $n \geq N$. We aim to show that $\P( d_{S_n} \neq 2 ) \to 0$, as $n \to \infty$. Define a sequence of random variables $Y_i$, $i = 1, 2, \dots$ where $Y_i = 1$ if $X_i$ is even and not prime, $Y_i = 0$ if $X_i$ is prime and $Y_i = -1$ if $X_i$ is odd and not prime. 
\[
\P(Y_i = 1) = \frac{t(n)-1}{n-1} > \P(Y_i = -1) = \frac{n - t(n) - \pi(n)}{n-1}
\]
Therefore, we can define a random walk $T_n := \sum_{i=1}^{m_n} Y_i$. If $T_n \geq 0$, then $d_{S_n} = 2$, and vice versa. Hence, $\P(d_{S_n} \neq 2) = \P(T_n < 0)$. We will use an approximation to this probability, via Hoeffding's inequality. We note that 
\[
\mu_n := \E[Y_i] = \frac{2\big(t(n)-1\big) + \pi(n)}{n-1} - 1,
\]
Additionally, standard properties of $\pi(n)$ imply that there exists some $N_0 \geq N$ such that
\[
\mu_n \geq \frac{1}{\log n} - \frac{1}{n} \geq \frac{1}{2\log n},
\]
for all $n \geq N_0$. We see that as $\mu_n > 0$, then
\[
\P(T_n < 0) = P(T_n/m_n < \mu_n - \mu_n) \leq e^{-\frac{m_n \mu_n^2}{2}}, 
\]
by Hoeffding's inequality and we know that for $n \geq N_0$ we have
\[
m_n \mu_n^2 \geq \frac{m_n}{8(\log n)^2} \to \infty, \quad n \to \infty,
\]
by the assumption on $m_n$. If $m_n = \omega\big( (\log n)^3 \big)$, then in particular there exists some $\delta > 0$ such that $m_n/(8 \log n)^2 \geq (1 + \delta) \log n$ for large enough $n$. As $\sum_{n=1}^{\infty} n^{-(1+\delta)} < \infty$, we have almost sure convergence by the Borel-Cantelli lemma. 
\end{proof}

In AdaBoost, given a sample $S$, we have at round $t \in \N$ a weight $w_t$ for our chosen hypothesis $h_t \in \mathcal{H}'$. We denote $w_t$ as $W_t$ when $S$ is seen as random as opposed to fixed/observed. Suppose that our instances are generated according to $\mathcal{D}_n$ and are labelled according to the prime labeling function $x \mapsto r(x)$. \emph{We suppose that $0$ is redefined as $-1$ in our indicator functions} to comport with the typical setting. One question is: as $n \to \infty$ (supposing that $m_n \to \infty$ as well) what is the stochastic behavior of the random weights $(W_{t,n}, {t \in \N})$? We have already established that the ERM threshold $d_{S_n}$ for the first round of AdaBoost becomes $2$ for large enough $n$ (with probability 1) in the proof of Proposition~\ref{e:wt_conv}. In the first round of AdaBoost, we set each $a_i = 1/{m_n}$ with $\epsilon_{1,n} = L_{S_n}(h_{S_n})$ and hence we get that 
\[
W_{1,n} = \frac{1}{2} \log \bigg( \frac{1}{\epsilon_{1,n}} - 1\bigg) \overset{\text{a.s.}}{\to} 0, \quad n \to \infty,
\]
because $\big |L_{S_n}(h_{S_n}) - L_{\mathcal{D}_n}(h_{S_n})\big | \overset{\text{a.s.}}{\to} 0, n \to \infty$ when $m_n = \omega\big((\log n)^3\big)$, by Hoeffding's inequality, the triangle inequality and Proposition~\ref{e:wt_conv}. We can extend this result quite a bit further. Before beginning, note that 
\[
\epsilon_{t,n} := \min_{h \in \mathcal{H}'} \sum_{i=1}^{m_n} D^{(t)}_i \ind{h(X_i) \neq r(X_i)},
\]
where $D_i^{(t)}$ are weights summing to 1 (defined at \eqref{e:dit}) and $h_t$ is the hypothesis in $\mathcal{H}'$ which minimizes $\epsilon_{t,n}$. Furthermore, $W_{t,n}$ is defined as 
\[
W_{t, n} := \frac{1}{2} \log \Big( \frac{1}{\epsilon_{t,n}} - 1\Big).
\]
\begin{thm} \label{t:all_wts}
Suppose that $m_n = \omega\big((\log n)^3\big)$ and that $W_{t,n}$ are the weights for the $t^{th}$ round AdaBoost hypothesis based on the sample $S_n$ and uniform distribution $\mathcal{D}_n$. Then we have for all $t \in \N$ that 
\[
W_{t,n} \overset{\text{a.s.}}{\to} 0, \quad n \to \infty.
\]
\end{thm}
\begin{proof}
We will define $D^{(t)}_i$ below. We have already established that $W_{1,n} \to 0$ a.s. as $n \to \infty$. We will now show that $\P\big(  \lim_{n \to \infty} W_{t,n} = 0 \big) = 1$ for each $t \in \N$. We proceed by induction (we have already demonstrated the base case) and suppose that for all $s < t+1$ that $\P\big( \lim_{n \to \infty} W_{s,n} = 0 \big) = 1$. We begin by noting that for any $1 \leq i \leq m_n$ that
\begin{equation}\label{e:dit}
D^{(t+1)}_i := \frac{D^{(t)}_ie^{-W_{t,n}r(X_i)h_t(X_i)}}{\sum_{i=1}^{m_n} D^{(t)}_ie^{-W_{t,n}r(X_i)h_t(X_i)}},
\end{equation}
where $D^{(1)}_i = 1/m_n$. Hence, we establish that 
\[
D^{(t)}_ie^{-2W_{t,n}} \leq D^{(t+1)}_i \leq D^{(t)}_ie^{2W_{t,n}}.
\]
Therefore, 
\[
\frac{1}{m_n} \exp\bigg( -2\sum_{j=1}^t W_{j,n} \bigg) \leq D^{(t+1)}_i \leq \frac{1}{m_n} \exp\bigg( 2\sum_{j=1}^t W_{j,n} \bigg),
\]
for all $n$, and
\[
\exp\bigg( -2\sum_{j=1}^t W_{j,n} \bigg) L_{S_n}(h_{S_n}) \leq \epsilon_{t+1,n} \leq \exp\bigg( 2\sum_{j=1}^t W_{j,n} \bigg) L_{S_n}(h_{S_n}),
\]
which implies that $\epsilon_{t+1,n} \overset{\text{a.s.}}{\to} 1/2$, by the induction hypothesis. Therefore, $W_{t+1,n} \overset{\text{a.s.}}{\to} 0$ as desired. 
\end{proof}

What Theorem~\ref{t:all_wts} says is that the output hypothesis of AdaBoost, 
\[
H^T_n(x) := \text{sign}\Bigg( \sum_{t=1}^{T} W_{t,n} h_t(x) \Bigg),
\]
for any fixed $T \in \N$ and a sufficiently large sample size $n$, is pretty much useless in learning the primality of the integers. At the very least, it is no better than a function that predicts all odd numbers are prime. 



\section{VC-dimension of finite arithmetic progressions}

We conclude by returning a variation of our original problem and considering the class $\mathcal{H}_k = \{h_{d,k}: d \in \X\}$, $k \in \X$ where $h_{d,k}(x) = 0$ if and only if $x \in \{2d, \dots, kd\}$. Suppose that $\mathcal{H}'_k$ is the restriction of $\mathcal{H}_k$ where $d$ is prime. These classes also cannot misidentify the primes, as with $\mathcal{H}$ and $\mathcal{H}'$. The VC-dimension of $\mathcal{H}_2$ and $\mathcal{H}'_2$ both equal 1 and this can be seen by noting that if $c_1 = 2d$ and $c_1 \neq c_2$, then $c_2 \neq 2d$ so that if $h_{d,2}(c_1) = 1$ then $h_{d,2}(c_2) = 0$. Consider the set $C= \{12, 18\}$. Then $(h_{6,3}(12), h_{6,3}(18)) = (0, 0)$, $(h_{4,3}(12), h_{4,3}(18)) = (0, 1)$, $(h_{9,3}(12), h_{9,3}(18)) = (1,0)$ and $(h_{7,3}(12), h_{7,3}(18)) = (1,1)$. Thus, $\VCdim(\mathcal{H}_3) = 2$. Suppose that $\mathcal{H}_4$ shatters a set $\{c_1, c_2, c_3\}$ of size 3. Then, potentially relabeling instances, $c_1 = 2d$, $c_2 = 3d$ and $c_4=4d$ for some $d$. However, we must also have $c_1 = 3d'$ and $c_3 = 4d'$ or $c_1 = 2d'$ and $c_3 = 3d'$, which is impossible. Hence, $\VCdim(\mathcal{H}_4) = 2$. 

As a final example, let us look at the VC-dimension of $\mathcal{H}'_3$. To shatter a set $\{c_1, c_2\}$, we must have $c_1 = 2p = 3q$ and $c_2 = 3p = 2r$ for some primes $q < p < r$. Therefore, $2 \mid p$ and $3 \mid p$, which is impossible as $p$ is prime. Hence $\VCdim(\mathcal{H}'_3) = 1$. As usual, we will denote the number of primes less than or equal to $x$ as $\pi(x)$. It will help to state a short lemma, similar to Corollary~\ref{c:pshat}.

\begin{lem} \label{l:factor}
For $\mathcal{H}_k$ to shatter a set $C = \{c_1, \dots, c_\ell\}$, it must be the case that each $c_i$ has at least $2^{\ell-1}$ distinct factors $d_1, \dots, d_{2^{\ell-1}}$, $d_i \in \X$ for $i = 1, 2, \dots, 2^{\ell-1}$. As result of the structure of $\mathcal{H}_k$, if shattering occurs, we must also have that $2^{\ell-1} \leq k-1$. 
\end{lem}

\begin{proof}
If $\mathcal{H}_k$ shatters $C$, then $h_{d,k}(c_i) = 0$ for $2^{\ell-1}$ distinct values of $d$. In other words $c_i = a_{i, 1}d_1 = \cdots = a_{i, m_i}d_m$, where $m_i \geq 2^{\ell-1}$ where all $a_{i,j} \leq k$ for $1 \leq j \leq m_i$. As the $d_j$ are distinct, so must the $a_{i,j}$ be. Therefore, we must have $m_i \leq k-1$. 
\end{proof}

Our final result gives a bound on the VC-dimension of $\mathcal{H}_k$ and $\mathcal{H}'_k$. This result is conceptually similar to Theorem 3.1 in \cite{learning_integer} which the VC-dimension of subsets of multiples of $d$ lying between $-n$ and $n$. 

\begin{prop}
\[
\big \lfloor \log_2 \pi ( \eta_k ) \big \rfloor \leq \VCdim(\mathcal{H}'_k) \leq \VCdim(\mathcal{H}_k) \leq \lceil \log_2(k-1) \rceil + 1
\]
where 
\[
\eta_k := \lfloor (\log_2 k)/2 \rfloor.
\]
\end{prop}

\begin{proof}
Suppose first that $\ell \geq \lceil \log_2(k-1) \rceil+2$. if $\mathcal{H}_k$ were to shatter $C = \{c_1, \dots, c_\ell\}$, then Lemma~\ref{l:factor} implies that $2^{\ell-1} \leq k-1$, or equivalently that 
\begin{align*}
\ell - 1 &\leq \log_2(k-1) \\
\Rightarrow \lceil \log_2(k-1) \rceil+1 &\leq \log_2(k-1),
\end{align*}
a contradiction. Hence, $\VCdim(\mathcal{H}_k) \leq \lceil \log_2(k-1) \rceil + 1$. 

Now consider a set of size $$\ell \leq \lfloor \log_2 \pi(n) \rfloor,$$ where $n \leq k$ so that there are at least $2^\ell \leq \pi(n)$ distinct primes less than or equal to $n$ in the set $\{2, \dots, k\}$. Enumerate these primes $p_1 < p_2 < \dots < p_{2^\ell} \leq n \leq k$ and define 
\[
c_i := \prod_{j=1}^{2^\ell} p_j^{\mathbf{1}_{A_j}(i)},
\]
where similar to the above $A_j$ are an enumeration of all the subsets of $\{1, \dots, \ell\}$. Now, consider some subset $A_m \subset \{1, \dots, \ell\}$. Then we have for $i \in A_m$ that
\[
c_i = p_m \prod_{\substack{1 \leq j \leq 2^\ell \\ j \neq m}} p_j^{\mathbf{1}_{A_j}(i)},
\]
so we are set if we can show that
\[
\prod_{\substack{1 \leq j \leq 2^\ell \\ j \neq m}} p_j^{\mathbf{1}_{A_j}(i)} \leq k.
\]
We have that 
\[
\prod_{\substack{1 \leq j \leq 2^\ell \\ j \neq m}} p_j^{\mathbf{1}_{A_j}(i)} \leq 2^{2n},
\]
by Theorem 415 in \cite{hardy_wright}. Thus, taking any $n \leq \lfloor \log_2 (k)/2 \rfloor$ will do. 
\end{proof}


\printbibliography

\end{document}